\documentclass[reqno]{amsart}

\usepackage{tikz}
\usetikzlibrary{shapes,arrows,patterns}
\usepackage{enumerate}
\usepackage{blkarray}
\usepackage{float}

\usepackage{amsthm,dsfont,mathrsfs,amsmath,amssymb,amsfonts,euscript,mathtools}
\usepackage{graphicx}

\usepackage[colorlinks=true,citecolor=blue,linkcolor=red,urlcolor=blue]{hyperref}

\newtheorem{theorem}{Theorem}

\newtheorem{lemma}[theorem]{Lemma}
\theoremstyle{remark}
\newtheorem{remark}[theorem]{Remark}
\newtheorem{example}[theorem]{Example}
\theoremstyle{definition}
\newtheorem{definition}[theorem]{Definition}

\begin{document}

\title[An Explicit Formula for the Eigenvectors of Acyclic Matrices and Weighted Trees]{An Explicit Formula for the Eigenvectors of Acyclic Matrices and Weighted Trees}%

\author{Asghar Bahmani}
\address{Department of Mathematics and Computer Science, Amirkabir University of Technology, 424, Hafez Ave., Tehran 15914, Iran, and School of Mathematics, Institute for Research in Fundamental Sciences (IPM), P.O. Box 19395-5746, Tehran, Iran.}
\email{asghar.bahmani@aut.ac.ir}
\author{Dariush Kiani}
\address{Department of Mathematics and Computer Science, Amirkabir University of Technology, 424, Hafez Ave., Tehran 15914, Iran, and School of Mathematics, Institute for Research in Fundamental Sciences (IPM), P.O. Box 19395-5746, Tehran, Iran.}
\email{dkiani@aut.ac.ir, dkiani7@gmail.com}


\begin{abstract}
Let $A$ be an acyclic symmetric matrix of order $n$. There is a weighted forest $F$ whose  adjacency matrix is $A$. In this paper, using some results on matching polynomials, we provide an explicit formula for  eigenvectors of $A$.
\end{abstract}
\maketitle

\section{Introduction}

For a positive integer $n$, we denote by $\text{\rm Sym}_{n}(\mathbb{R})$, the set of real symmetric matrices of order $n$. We denote by $[n]$ the set $\{1,\ldots,n\}$.
For an arbitrary symmetric matrix $A$,  $\phi(A,\lambda)=|\lambda \mathbb{I}-A|$ is the characteristic polynomial of $A$.
For an index set $U$ and a symmetric matrix $A$, $A-U$ is the matrix that is obtained from $A$ by removing rows and columns corresponding to $U$.

\begin{definition}[\textbf{Acyclic matrix}]
Let $A$ be a symmetric real matrix. We say $A$ is acyclic if $A$ is an adjacency matrix of a weighted forest with real nonzero weights on edges (and possibly on vertices). If $A$ is irreducible, then it is an adjacency matrix of a weighted tree.
\end{definition}

Eigenvectors of matrices in linear algebra and eigenvectors of matrices associated with graphs have many applications in algebraic graph theory such as drawing of graphs, partitioning of sparse matrices, partitioning networks, and optimization. See for example \cite{Ko}, \cite{PSL}, \cite{SR}, and \cite{Sp}.

There are several theorems such that for some matrices give  formulas for the eigenvectors of their eigenvalues.
Some of these results are

\begin{enumerate}[i.]

\item {Cayley Graphs} \cite[Section 1.4]{BH},

\item {Circulant Matrices} \cite{Gr},

\item {Paths and Cycles} \cite[Section 1.4]{BH},

\item {Tridiagonal Matrices} \cite[11.4]{Pr},\cite{Kou},

\item {Cartesian, Kronecker, Strong Products of Graphs} \cite[Section 1.4]{BH}.

\end{enumerate}

In this paper we use some theorems of matching polynomials of graphs  to find a formula for the eigenvectors of acyclic symmetric matrices and weighted trees.

The main results of this paper are Theorems \ref{simpleeig} and \ref{multipleeig}.

\section{Preliminaries}
In this section, we review some results on star sets of symmetric matrices and matching polynomials that we need in this paper.

\subsection{Star Sets}

Let $A\in \text{\rm Sym}_{n}(\mathbb{R})$ and $\lambda$ be an eigenvalue of $A$ of multiplicity $k$. A set $U\subseteq [n]$ is a \textit{star set} for $\lambda$ (or $\lambda$-star set) of  $A$ if $|U|=k$ and $\lambda$  is not an eigenvalue of the submatrix of $A$ obtained by removing rows and columns with index in $U$. It is known that for every eigenvalue $\lambda$ there exists a $\lambda$-star set \cite{CRS}.

By definition of star set, it is clear that for an index set $U$ such that $|U|=k$, $\phi(A-U,\lambda)\neq 0$ if and only if  $U$ is a $\lambda$-star set. The following lemma characterizes the star sets for simple eigenvalues.

\begin{lemma}[\textbf{Star set of a simple eigenvalue}]\label{starsetsimple}
Let $A$ be a symmetric matrix with a simple eigenvalue $\lambda$  and a $\lambda$-eigenvector $\boldsymbol{\alpha}$. If $u$ is a row (and column) index of $A$, then $\{u\}$ is a $\lambda$-star set if and only if  $\boldsymbol{\alpha}(u)\neq 0$.
\end{lemma}

\begin{proof}
Suppose that  $\{u\}$ is a $\lambda$-star set and assume by contradiction that $\boldsymbol{\alpha}(u)=0$. Now, by Lemma 7 
of \cite{BK2}, $\lambda$ is an eigenvalue of $A-u$, too, a contradiction by assumption ($\{u\}$ is a $\lambda$-star set). 

Now, suppose that  $\boldsymbol{\alpha}(u)\neq 0$.
Suppose that $I$ is the set of row indices of $A$ and  
$A=\begin{blockarray}{cc|c}
\begin{block}{c(c|c)}
\text{\footnotesize $u$ } & a  & \boldsymbol{x}^{T}   \\\cline{1-3}
\text{\tiny $I-\{u\}$ }&\boldsymbol{x}  &  M  \\
\end{block}
\end{blockarray}
$, for some $M$. It is sufficient to show that  $m_{M}(\lambda)=0$. Suppose, on the contrary, $M$ has a $\lambda$-eigenvector $\boldsymbol{\beta}$.
We have two cases:

\textbf{Case 1.}  $\boldsymbol{x}^{T}\boldsymbol{\beta}=0$: So, $\boldsymbol{y}$ is a $\lambda$-eigenvector of  $A$, where $\boldsymbol{y}(v)=
\begin{cases}
\boldsymbol{\beta}(v)& v\neq u,\\
0 & v=u
\end{cases}$.
Since $\boldsymbol{\alpha}(u)\neq 0$, the vectors $\boldsymbol{\alpha}$ and $\boldsymbol{y}$ are independent and we have a contradiction with $m_{A}(\lambda)=1$. 

\textbf{Case 2.} $\boldsymbol{x}^{T}\boldsymbol{\beta}\neq 0$: So,
\[
0=(\lambda\mathbb{I}-A)\boldsymbol{\alpha}\Rightarrow \boldsymbol{\alpha}(u)\,\boldsymbol{x}=(\lambda\mathbb{I}-M)\boldsymbol{\alpha}_{|\text{\tiny $I-\{u\}$ }}\Rightarrow\]\[
\boldsymbol{\alpha}(u)\boldsymbol{\beta}^{T}\boldsymbol{x}=\boldsymbol{\beta}^{T}(\lambda\mathbb{I}-M)\boldsymbol{\alpha}_{|\text{\tiny $I-\{u\}$ }}=0 \xRightarrow{\boldsymbol{\beta}^{T}\boldsymbol{x}\neq 0}
\boldsymbol{\alpha}(u)=0,
\]
 and we have a contradiction. This completes the proof.

\end{proof}

\subsection{{Matching Polynomial}}
We need these definitions and theorems on matching polynomials.
For a  given (weighted) graph $G$ and two its vertices $u$ and $v$, the set of all paths from $u$ to $v$ is denoted by $\mathcal{P}(u,v)$. If $H$ is a subgraph of $G$ and $u\notin V(H)$, the set of all paths from $u$ to some vertex in $H$, say $z$,  without  visiting  $H$  before $z$, is denoted by  $\mathcal{P}(u,H)$.
When $G$ is a tree, we denote by $P_{u,v}$ the path between $u$ and $v$.

\begin{definition}[\textbf{Matching}]
Let $G$ be a simple graph and $k$ be a nonnegative integer. A set of $k$ edges of $G$ is called a matching (or $k$-matching), if and only if no two edges have a common vertex.
\end{definition}
One of the subjects on matchings is algebraic approach and considering generating polynomials. Matching polynomials have been introduced and considered in various areas such as physics and chemistry, see \cite{CDGT} and \cite{Go} for background history. We need the following definitions. 

\begin{definition}[\textbf{Matching polynomials for simple graphs}]
Let $n\in\mathbb{N}$ and $G$ be a simple graph on $n$ vertices. For a nonnegative integer $k$, we denote by $p(G,k)$ the number of  $k$-matching of $G$. The polynomial $\mu(G,x)=\displaystyle \sum_{k=0}^{\lfloor\frac{n}{2}\rfloor} (-1)^{k}p(G,k)x^{n-2k}$ is the matching polynomial of $G$. 
\end{definition}

Assume that $G$ is a simple graph and $w:E(G)\rightarrow \mathbb{R}$ is a weight function on $G$. We denote this weighted graph by $G_w$.  For an arbitrary $k$-matching $\mathfrak{m}$ of $G$, we define its weight as:
\[
W(\mathfrak{m})=\prod_{e\in \mathfrak{m} }w(e),\, W(\emptyset)=1.
\]
Now, we have the definition of matching polynomial of a weighted graph with weights on its edges.

\begin{definition}[\textbf{Weights on edges}]
Let $n\in \mathbb{N}$ and $G$ be a simple graph with a weight function $w:E(G)\rightarrow \mathbb{R}$.
For a nonnegative integer $k$, 
\[p(G_w,k)=\displaystyle\sum_{\mathfrak{m}:k\text{-matching}}(W(\mathfrak{m}))^{2}\]
is the sum of the weighted $k$-matchings of $G_w$. The polynomial 
\[\mu(G_w,x)=\displaystyle \sum_{k=0}^{\lfloor\frac{n}{2}\rfloor} (-1)^{k}p(G_w,k)x^{n-2k}\]
 is the matching polynomial of $G_w$. 
\end{definition}
For weighted graphs with weights on edges and vertices, we have the following generalized case:

\begin{definition}[\textbf{Weights on edges and vertices}]
Let $n\in \mathbb{N}$ and $G$ be a simple graph with a weight function $w:V(G)\cup E(G)\rightarrow \mathbb{R}$.
The polynomial 
\[\mu(G_w,x)=\displaystyle\sum_{\mathfrak{m}}(-1)^{|\mathfrak{m}|}(W(\mathfrak{m}))^{2}\prod_{u \text{ is not in } \mathfrak{m}}(x-w(u))\]
 is the matching polynomial of $G_w$, where 
$|\mathfrak{m}|$ is the number of edges of $\mathfrak{m}$.
\end{definition}
In case of empty graph, we define $\mu(\emptyset,x)=1$.
For more details, see
\cite{CDGT}, \cite{Go}, \cite{Go2}, \cite{HL},  and  \cite{KW}.
We know that the roots of matching polynomials are real \cite[Theorem 4.27]{CDGT},\cite[Theorem 6.2]{Go},\cite[Theorem 4.2]{HL},\cite[Theorem 2.17]{KW}.
The following theorem, is the  key relationship between the matching polynomial and the characteristic polynomial of the adjacency matrix  of a weighted forest.

\begin{theorem}{\rm \cite[Theorem 4.26]{CDGT},\cite[ Corollary 2.9]{KW}}\label{matchchar}
Let $F$ be a weighted Forest. Then  $\phi(F,x)=\mu(F,x)$.
\end{theorem}
The relations below, are two important relations on matching polynomials.

\begin{theorem}{\rm \cite[Section 4.2]{CDGT},\cite[Relation 4.3]{HL},\cite[Theorem 2.1]{KW}}\label{relationsheli}
Let $G_w$ be a weighted graph ($w:V(G)\cup E(G) \rightarrow \mathbb{R}$). We have the following relations.
\begin{enumerate}[i.]

\item $\displaystyle \frac{d\mu(G_w,x)}{dx}=\displaystyle\sum_{v\in V(G)}\mu(G_w-v,x)$,
\item For every vertex $v$,
\begin{equation}\label{heli}
\mu(G_w,x)=(x-w(v))\mu(G_w -v,x)-\displaystyle\sum_{z\neq v}(w(zv))^{2}\mu(G_w-\{v,z\},x),
\end{equation}
where $w(zv)$ is the weight of the edge $zv$.
 \end{enumerate}
\end{theorem}

For an arbitrary path $P$, we define its weight as
\[
W(P)=\prod_{e\in E(P) }w(e) \text{ and }  W(P)=1  \text{ if $P$ is a vertex}.
\]
The following important theorem gives a relationship between paths of a graph and the matching polynomial.
\begin{theorem}{\rm \cite[Theorem 6.3]{HL}}\label{helis}
Let $G_w$ be a weighted graph ($w:V(G)\cup E(G) \rightarrow \mathbb{R}$). If $u\in V(G)$ and $H$ is a subgraph of $G$ such that $u\notin V(H)$, then we have
\[
\mu(G_w-u,x)\mu(G_w-H,x)-\mu(G_w,x)\mu(G_w-\{u,H\},x)=\]\[
\sum_{P\in \mathcal{P}(u,H)}(W(P))^{2}\mu(G_w-P,x)\mu(G_w-\{H,P\},x).
\]
\end{theorem}
This theorem in  \cite{HL} is stated for weighted graphs without weight on vertices. Also, in our notations, $(W(P))^{2}$ has been used instead of $W(P)$ in \cite[Theorem 6.3]{HL}.
The following proof is for our general case.

\begin{proof}
We prove the statement by induction on $V(G)-V(H)-\{u\}$.
For brevity, for any subgraph or vertex subset $S$ of $G_w$, we use this notation:
\[\mu_{\{S\}}:=\mu(G_w-S,x).\]
In the case that $H$ is empty or $V(G)=\{u\}$, it is easy to see that the statement is true. So, suppose that $H$ is nonempty.

From (\ref{heli}), we have:
{
\begin{equation}\label{hel1}
\mu(G_w,x)=(x-w(u))\mu(G_w -u,x)-\displaystyle\sum_{v\neq u}(w(uv))^{2}\mu(G_w-\{v,u\},x),
\end{equation}
\begin{equation}\label{hel2}
\mu(G_w-H,x)=(x-w(u))\mu_{\{u,H\}}-\displaystyle\sum_{\substack{v\neq u \\ v\notin V(H)}}(w(uv))^{2}\mu(G_w-\{v,u,H\},x).
\end{equation}
}
So, from (\ref{hel1}), we conclude
{
\begin{multline}\label{hel3}
\mu(G_w,x)\mu(G_w-\{u,H\},x)=
((x-w(u))\mu_{u}-\displaystyle\sum_{v\neq u}(w(uv))^{2}\mu_{\{v,u\}})\mu_{\{u,H\}}= \\
(x-w(u))\mu_{u}\mu_{\{u,H\}}-\displaystyle\sum_{\substack{v\neq u \\ v\notin V(H)}}(w(uv))^{2}\mu_{\{v,u\}}\mu_{\{u,H\}}-\displaystyle\sum_{v\in V(H)}(w(uv))^{2}\mu_{\{v,u\}}\mu_{\{u,H\}}.
\end{multline}
}
Also, from (\ref{hel2}), we have
{
\begin{equation}\label{hel4}
\mu(G_w-u,x)\mu(G_w-H,x)= 
(x-w(u))\mu_{u}\mu_{\{u,H\}}-\displaystyle\sum_{\substack{v\neq u \\ v\notin V(H)}}(w(uv))^{2}\mu_{u}\mu_{\{v,u,H\}}.
\end{equation}
}
By using (\ref{hel3}) and (\ref{hel4}), we have
{
\begin{multline}\label{indu}
\mu(G_w-u,x)\mu(G_w-H,x)-\mu(G_w,x)\mu(G_w-\{u,H\},x)=\\
\displaystyle\sum_{v\in V(H)}(w(uv))^{2}\mu_{\{v,u\}}\mu_{\{u,H\}}+\displaystyle\sum_{\substack{v\neq u \\ v\notin V(H)}}(w(uv))^{2}(\mu_{\{v,u\}}\mu_{\{u,H\}}-\mu_{u}\mu_{\{v,u,H\}}).
\end{multline}
}
Now, for the base case of induction, suppose that $V(G)=V(H)\cup\{u\}$. From (\ref{indu}) one can conclude
{
\[
\mu(G_w-u,x)\mu(G_w-H,x)-\mu(G_w,x)\mu(G_w-\{u,H\},x)=
\displaystyle\sum_{v\in V(H)}(w(uv))^{2}\mu_{\{v,u\}}\mu_{\{u,H\}}.
\]
}
Since every path from $u$ to $H$ is an edge such as $uv$ where $v\in V(H)$, we have: 
{
\begin{multline*}
\mu(G_w-u,x)\mu(G_w-H,x)-\mu(G_w,x)\mu(G_w-\{u,H\},x)=\\
\displaystyle\sum_{v\in V(H)}(w(uv))^{2}\mu_{\{v,u\}}\mu_{\{u,H\}}=
\sum_{P\in \mathcal{P}(u,H)}(W(P))^{2}\mu(G_w-P,x)\mu(G_w-\{H,P\},x).
\end{multline*}
}
Hence, the statement is true for the base case of induction. Therefore, for $V(H)\cup\{u\}\varsubsetneqq V(G)$, suppose that the statement is true for $G_w-u$ and an arbitrary vertex $v$,  $v\notin V(H)$, and we have
{
\[
\mu_{\{v,u\}}\mu_{\{u,H\}}-\mu_{u}\mu{\{v,u,H\}}=\sum_{\substack{P\in \mathcal{P}(v,H) \\ P\, in\, G_w-u }}(W(P))^{2}\mu(G_w-u-P,x)\mu(G_w-\{H,u,P\},x).
\]
}
So, by substituting the relations above, in (\ref{indu}) we have
\begin{multline*}
\displaystyle\sum_{v\in V(H)}(w(uv))^{2}\mu_{\{v,u\}}\mu_{\{u,H\}}+\\
\displaystyle\sum_{\substack{v\neq u \\ v\notin V(H)}}(w(uv))^{2}(\sum_{\substack{P\in \mathcal{P}(v,H) \\ P\, in \, G_w-u}}(W(P))^{2}\mu(G_w-u-P,x)\mu(G_w-\{H,u,P\},x))=\\
\sum_{P\in \mathcal{P}(u,H)}(W(P))^{2}\mu(G_w-P,x)(G_w-\{H,P\},x),
\end{multline*}
and hence the statement follows.

\end{proof} 

\section{Eigenvectors of Acyclic Matrices and Trees }

In this section, we give an explicit formula for the eigenvectors of acyclic matrices and weighted trees. 
When we say two indices $u$ and $v$ of $A$ are adjacent or there exists a path between them, we mean that in the graph associated to $A$, they are adjacent or  there exists a path between them, respectively.
\subsection{Simple Eigenvalue}
Now, in the following theorem, we give a formula for the eigenvectors of simple eigenvalues of  acyclic matrices.
\begin{theorem}\label{simpleeig}
Suppose that $n\in\mathbb{N}$ and $A$ is an $n\times n$ acyclic symmetric matrix. If  $\lambda$ is a simple eigenvalue of $A$ and $u$ is an index such that $\phi(A-u,\lambda)\neq 0$, then 
 $\boldsymbol{\alpha}$ is a $\lambda$-eigenvector, where
 \[
 \boldsymbol{\alpha}(v)=
 \begin{cases}
 W(P_{u,v})\phi(A-P_{u,v},\lambda) & \mathcal{P}(u,v)\neq \emptyset,\\
 0 & \text{otherwise}.
 \end{cases}
 \] 

\end{theorem}

\begin{proof}
It is sufficient to prove the statement for the tree that contains $u$ in the forest.
Since $\boldsymbol{\alpha}(u)=\phi(A-u,\lambda)\neq 0$, the given vector $\boldsymbol{\alpha}$ is not zero. It is sufficient to show that for every index $v$, we have 
\begin{equation}\label{eigvectorrel}
(\lambda-w(v))\boldsymbol{\alpha}(v)=\sum_{z\sim v }w(vz)\boldsymbol{\alpha}(z).
\end{equation}

From (\ref{heli})  and Theorem \ref{matchchar}, by putting $x=\lambda$ and $G_w=A$, we conclude
\[
0=\phi(A,\lambda)=(\lambda-w(v))\phi(A-v,\lambda)-\displaystyle\sum_{z\neq v}(w(zv))^{2}\phi(A-\{v,z\},\lambda)
\]
and hence,
\begin{equation}\label{eigvectorrel1}
(\lambda-w(v))\phi(A-v,\lambda)=\displaystyle\sum_{z\sim v}(w(zv))^{2}\phi(A-\{v,z\},\lambda).
\end{equation}
Now, for two vertices $v$ and $u$, in Theorem \ref{helis}, put $G_w=A$, $x=\lambda$, and $H=\{v\}$. Therefore, we have
\[
\phi(A-u,\lambda)\phi(A-v,\lambda)-\phi(A,\lambda)\phi(A-\{u,v\},\lambda)=(W(P_{u,v}))^{2}\phi(A-P_{u,v},\lambda)\phi(A-P_{u,v},\lambda).\] 
Since  $\phi(A,\lambda)=0$,
\begin{equation}\label{eigvectorrel2}
\phi(A-v,\lambda)=(\frac{W(P_{u,v})\phi(A-P_{u,v},\lambda)}{\phi(A-u,\lambda)})W(P_{u,v})\phi(A-P_{u,v},\lambda).
\end{equation}
Now, for two arbitrary adjacent vertices $v$ and $z$, in Theorem \ref{helis}, put
\begin{center}
 $G_w=A$, $x=\lambda$ and
 $H=\text{(induced subgraph by $v$ and $z$)}=e(zv)$.
\end{center}
 
 We have two following cases:
 
\textbf{Case 1.} Suppose the only path in the forest corresponding to $A$, from $u$ to $H$ which intersects $H$ in a vertex, meets $H$ at $v$. So, we have
{ \[
\phi(A-u,\lambda)\phi(A-H,\lambda)-\phi(A,\lambda)\phi(A-\{u,H\},\lambda)=\]\[(W(P_{u,v}))^{2}\phi(A-P_{u,v},\lambda)\phi(A-\{H,P_{u,v}\},\lambda).\] }
Since $\phi(A,\lambda)=0$ and $A-\{H,P_{u,v}\}=A-P_{u,z}$, we have
\[
\phi(A-u,\lambda)\phi(A-\{v,z\},\lambda)=(W(P_{u,v}))^{2}\phi(A-P_{u,v},\lambda)\phi(A-P_{u,z},\lambda),
\]
and we conclude that
\[
\phi(A-\{v,z\},\lambda)=(\frac{W(P_{u,v})\phi(A-P_{u,v},\lambda)}{\phi(A-u,\lambda)})W(P_{u,v})\phi(A-P_{u,z},\lambda),
\]
and because $w(zv)W(P_{u,v})=W(P_{u,z})$,
\[
(w(zv))^{2}\phi(A-\{v,z\},\lambda)=(\frac{W(P_{u,v})\phi(A-P_{u,v},\lambda)}{\phi(A-u,\lambda)})w(zv)(W(P_{u,z})\phi(A-P_{u,z},\lambda)).\]

\textbf{Case 2.} Suppose the only path in the forest corresponding to $A$, from $u$ to $H$ which intersects $H$ in a vertex, meets $H$ at $z$. So, we have
{ \[
\phi(A-u,\lambda)\phi(A-H,\lambda)-\phi(A,\lambda)\phi(A-\{u,H\},\lambda)=\]\[(W(P_{u,z}))^{2}\phi(A-P_{u,z},\lambda)\phi(A-\{H,P_{u,z}\},\lambda).\] }
Since $\phi(A,\lambda)=0$ and $A-\{H,P_{u,z}\}=A-P_{u,v}$, we have
\[
\phi(A-u,\lambda)\phi(A-\{v,z\},\lambda)=(W(P_{u,z}))^{2}\phi(A-P_{u,z},\lambda)\phi(A-P_{u,v},\lambda),
\]
and because $W(P_{u,z})=\frac{W(P_{u,v})}{w(zv)}$,
\[
\phi(A-\{v,z\},\lambda)=(\frac{W(P_{u,v})\phi(A-P_{u,v},\lambda)}{w(zv)\phi(A-u,\lambda)})W(P_{u,z})\phi(A-P_{u,z},\lambda),
\]
and hence,
\[
(w(zv))^{2}\phi(A-\{v,z\},\lambda)=(\frac{W(P_{u,v})\phi(A-P_{u,v},\lambda)}{\phi(A-u,\lambda)})w(zv)(W(P_{u,z})\phi(A-P_{u,z},\lambda)).\]
Therefore, in both cases we have
\begin{equation}\label{eigvectorrel3}
(w(zv))^{2}\phi(A-\{v,z\},\lambda)=(\frac{W(P_{u,v})\phi(A-P_{u,v},\lambda)}{\phi(A-u,\lambda)})w(zv)(W(P_{u,z})\phi(A-P_{u,z},\lambda)).
\end{equation}
By putting (\ref{eigvectorrel2}) and (\ref{eigvectorrel3}) in (\ref{eigvectorrel1}), we have
\begin{align*}
(\lambda-w(v))(\frac{W(P_{u,v})\phi(A-P_{u,v},\lambda)}{\phi(A-u,\lambda)})W(P_{u,v})\phi(A-P_{u,v},\lambda)=\\
\displaystyle\sum_{z\sim v}(\frac{W(P_{u,v})\phi(A-P_{u,v},\lambda)}{\phi(A-u,\lambda)})w(zv)(W(P_{u,z})\phi(A-P_{u,z},\lambda))
\end{align*}
and by simplifying relation above, one can conclude 
\[
(\lambda-w(v))W(P_{u,v})\phi(A-P_{u,v},\lambda)=\displaystyle\sum_{z\sim v}w(zv)(W(P_{u,z})\phi(A-P_{u,z},\lambda)).
\]
Therefore, we obtain the relation (\ref{eigvectorrel}) and the proof is complete.
\end{proof}

\begin{remark}\label{pendant}
\textbf{Pendant vertices and star sets}:
The index $u$ in Theorem \ref{simpleeig}, concerning Lemma \ref{starsetsimple}, is a member of a star set. By theorems of star sets or by Theorem \ref{relationsheli} we conclude that there exists always such index.
One of the problems on star sets is finding a star set. In the case of weighted forests (with at least two vertices) and a simple eigenvalue, since the corresponding eigenvector is nonzero on at least two pendant vertices, the pendant vertices are good choices to find a star set.
\end{remark}
In the relation (\ref{eigvectorrel2}) for the root $\lambda$ we have
\begin{equation}\label{uvphi}
\phi(A-u,\lambda)\phi(A-v,\lambda)=(W(P_{u,v})\phi(A-P_{u,v},\lambda))^{2}.
\end{equation}
Thus, for every vertex $v$, $\phi(A-v,\lambda)$ is zero or has the same sign with $\phi(A-u,\lambda)$. 
Therefore, we have
\[
\sqrt{|\phi(A-u,\lambda)|}\sqrt{|\phi(A-v,\lambda)|}=|W(P_{u,v})\phi(A-P_{u,v},\lambda)|.
\]
By summing relation (\ref{uvphi}) over all vertices, we have
\[
\sum_{v}(W(P_{u,v})\phi(A-P_{u,v},\lambda))^{2}=\phi(A-u,\lambda)\sum_{v}\phi(A-v,\lambda)=\phi(A-u,\lambda)\phi^{'}(A,\lambda).\] 
The last equality is obtained from Theorem \ref{relationsheli}. Therefore, we have the following theorem.
\begin{theorem}
Let $n\in\mathbb{N}$ and $A\in \text{\rm Sym}_{n}(\mathbb{R})$ be an acyclic matrix with a simple eigenvalue $\lambda$. If $u$ is an index of $A$ such that $\phi(A-u,\lambda)\neq 0$, then there exist a unit  $\lambda$-eigenvector $\boldsymbol{\beta}$ and  a   $\lambda$-eigenvector $\boldsymbol{\gamma}$ such that for every $v\in [n]$ 
 \[
 \boldsymbol{\beta}(v)=
 \begin{cases}
\frac{W(P_{u,v})\phi(A-P_{u,v},\lambda)}{\sqrt{\phi(A-u,\lambda)\phi^{'}(A,\lambda)}}  & \mathcal{P}(u,v)\neq \emptyset,\\
 0 & \text{otherwise},
 \end{cases}\text{ and }
|\boldsymbol{\gamma}(v)|=
\sqrt{|\phi(A-v,\lambda)|}.
 \]
\end{theorem}

\begin{example}\label{examvector}
Suppose that $A$ is the following acyclic symmetric matrix and $\lambda$ is a simple eigenvalue of $A$. 
If for the vertex $1$ with weight $w_1$, we have $\phi(A-\{1\},\lambda)\neq0$, namely $1$ is a member of a star set,
then for the array  of the $\lambda$-eigenvector $\boldsymbol{\alpha}$ corresponding to the vertex $3$, we remove the path from $1$ to $3$ and we have the graph in Figure \ref{exampath13}. Hence, by Theorem \ref{simpleeig} we have
\[\boldsymbol{\alpha}(3)=W(P_{1,3})\phi(A-P_{1,3},\lambda)=fg\phi(A-P_{1,3},\lambda)=fg|\lambda\mathbb{I}-(A-P_{1,3})|.\]
{
\[
A=\begin{pmatrix}
w_1 & f & 0 & 0 & 0 & 0 & 0 & 0 & 0 & n \\ 
f &  w_2&g& 0 & 0 & j & 0 & 0 & 0 & 0 \\ 
0 & g & w_3&h  & 0 & 0 & 0 & 0 & 0 & 0 \\ 
0 & 0 & h & w_4 & i & 0 & 0 & 0 & 0 & 0\\
0 & 0 & 0 & i & w_5& 0 & 0 & 0 & 0 & 0 \\ 
0 & j & 0 & 0 & 0 &w_6 &k & 0 & m & 0 \\ 
0 & 0 & 0 & 0 & 0 & k &w_7& l & 0 & 0 \\ 
0 & 0 & 0 & 0 & 0 & 0 & l &w_8 & 0 & 0 \\ 
0 & 0 & 0 & 0 & 0 & m & 0 & 0 & w_9 & 0 \\ 
n & 0 & 0 & 0 & 0 & 0 & 0 & 0 & 0 & w_{10} 
\end{pmatrix} ,\,\,
A-P_{1,3}=\left(\begin{array}{cc|cccc|c}
 w_4 & i & 0 & 0 & 0 & 0 & 0\\
 i & w_5& 0 & 0 & 0 & 0 & 0 \\\cline{1-7} 
 0 & 0 &w_6 &k & 0 & m & 0 \\ 
 0 & 0 & k &w_7& l & 0 & 0 \\ 
 0 & 0 & 0 & l &w_8 & 0 & 0 \\ 
 0 & 0 & m & 0 & 0 & w_9 & 0 \\ \cline{1-7}
 0 & 0 & 0 & 0 & 0 & 0 & w_{10} 
\end{array}\right) 
\]

}
\begin{figure}[H]
\centering
\begin{tikzpicture}[scale=.8,line width=.8 pt ,line cap=round,line join=round,>=triangle 45,x=1cm,y=1.0cm]
\draw (1.7,5)-- (0.5,4.);
\draw  (0.5,4.)-- (-1.6,2.4);
\draw (-1.6,2.4)-- (-1.6,0.6);
\draw (-1.6,0.6)-- (-1.6,-1.1);
\draw (0.5,4.)-- (1.7,2.4);
\draw (1.7,2.4)-- (1.7,0.5);
\draw (1.7,0.5)-- (1.7,-1.1);
\draw (1.7,2.4)-- (4.,2.4);
\draw  (1.7,5)-- (4,5);
\begin{scriptsize}
\draw [fill=red] (1.7,5) circle (3.pt);
\draw[color=blue] (1.6,5.4) node {$w_1$};
\draw [fill=blue] (0.5,4.) circle (1.5pt);
\draw[color=blue] (0.38,4.36) node {$w_2$};
\draw[color=black] (0.8,4.7) node {$f$};
\draw [fill=blue] (-1.6,2.4) circle (1.5pt);
\draw [] (-1.6,2.4) circle (3.pt);
\draw[color=blue] (-1.44,2.7) node {$w_3$};
\draw[color=black] (-0.84,3.58) node {$g$};
\draw [fill=blue] (-1.6,0.6) circle (1.5pt);
\draw[color=blue] (-1.4,0.96) node {$w_4$};
\draw[color=black] (-1.88,1.66) node {$h$};
\draw [fill=blue] (-1.6,-1.1) circle (1.5pt);
\draw[color=blue] (-1.4,-0.76) node {$w_5$};
\draw[color=black] (-1.92,-0.08) node {$i$};
\draw [fill=blue] (1.7,2.4) circle (1.5pt);
\draw[color=blue] (1.86,2.6) node {$w_6$};
\draw[color=black] (0.8,3.22) node {$j$};
\draw [fill=blue] (1.7,0.54) circle (1.5pt);
\draw[color=blue] (1.94,0.9) node {$w_7$};
\draw[color=black] (1.5,1.7) node {$k$};
\draw [fill=blue] (1.7,-1.1) circle (1.5pt);
\draw[color=blue] (2.02,-0.84) node {$w_8$};
\draw[color=black] (1.58,-0.18) node {$l$};
\draw [fill=blue] (4.,2.4) circle (1.5pt);
\draw[color=blue] (4.3,2.6) node {$w_9$};
\draw[color=black] (3.,2.3) node {$m$};
\draw [fill=blue] (4,5.) circle (1.5pt);
\draw[color=blue] (4.1,5.4) node {$w_{10}$};
\draw[color=black] (2.8,4.86) node {$n$};
\end{scriptsize}

\draw[] (5,4) node {$\Longrightarrow$};

\draw[style=dashed] (8.7,5)-- (7.5,4.);
\draw [style=dashed]  (7.5,4.)-- (5.4,2.4);
\draw [style=dotted](5.4,2.4)-- (5.4,0.6);
\draw (5.4,0.6)-- (5.4,-1.1);
\draw [style=dotted] (7.5,4.)-- (8.7,2.4);
\draw (8.7,2.4)-- (8.7,0.5);
\draw (8.7,0.5)-- (8.7,-1.1);
\draw (8.7,2.4)-- (11,2.4);
\draw [style=dotted] (8.7,5)-- (11,5);
\begin{scriptsize}
\draw [color=red] (8.7,5) circle (3.pt);
\draw[color=blue] (8.6,5.4) node {$w_1$};
\draw [] (7.5,4.) circle (1.5pt);
\draw[color=blue] (7.38,4.36) node {$w_2$};
\draw[color=black] (7.8,4.7) node {$f$};
\draw [color=blue] (5.4,2.4) circle (1.5pt);
\draw [] (5.4,2.4) circle (3.pt);
\draw[color=blue] (5.56,2.7) node {$w_3$};
\draw[color=black] (6.16,3.58) node {$g$};
\draw [fill=blue] (5.4,0.6) circle (1.5pt);
\draw[color=blue] (5.6,0.96) node {$w_4$};
\draw[color=black] (5.1,1.66) node {$h$};
\draw [fill=blue] (5.4,-1.1) circle (1.5pt);
\draw[color=blue] (5.6,-0.76) node {$w_5$};
\draw[color=black] (5.1,-0.08) node {$i$};
\draw [fill=blue] (8.7,2.4) circle (1.5pt);
\draw[color=blue] (8.86,2.6) node {$w_6$};
\draw[color=black] (7.8,3.22) node {$j$};
\draw [fill=blue] (8.7,0.54) circle (1.5pt);
\draw[color=blue] (8.94,0.9) node {$w_7$};
\draw[color=black] (8.5,1.7) node {$k$};
\draw [fill=blue] (8.7,-1.1) circle (1.5pt);
\draw[color=blue] (9.02,-0.84) node {$w_8$};
\draw[color=black] (8.58,-0.18) node {$l$};
\draw [fill=blue] (11.,2.4) circle (1.5pt);
\draw[color=blue] (11.3,2.6) node {$w_9$};
\draw[color=black] (10,2.3) node {$m$};
\draw [fill=blue] (11,5.) circle (1.5pt);
\draw[color=blue] (11.1,5.4) node {$w_{10}$};
\draw[color=black] (9.8,4.86) node {$n$};
\end{scriptsize}
\end{tikzpicture}
\caption{Trees of Example \ref{examvector}}
\label{exampath13}
\end{figure}
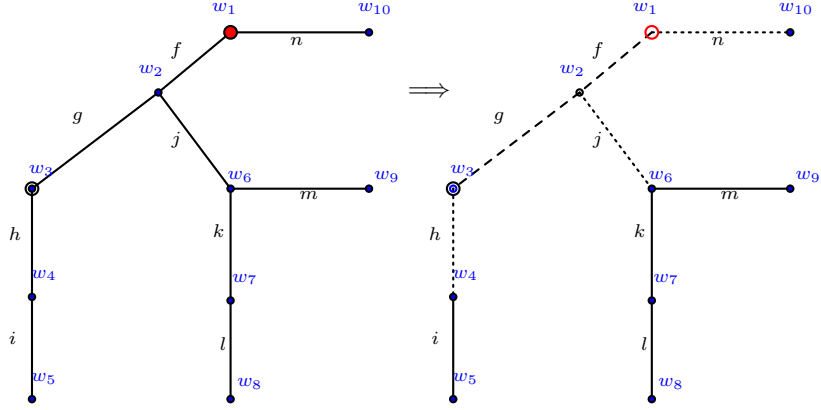
\end{example}

Before the next examples, we state the following theorem for path graphs.

\begin{theorem}[\textbf{Path Eigenvectors} ]{\rm \cite[Section 1.4]{BH}}\label{pathvec}
Let $n\in\mathbb{N}$. For the path $P_{n}$ we have:
\begin{enumerate}[i.]
\item
Adjacency: For every $k\in [n]$, $\lambda_{k}=\displaystyle 2\cos (\frac{k\pi}{n+1})$   is an adjacency eigenvalue of $P_n$ with corresponding eigenvector $\boldsymbol{v}_k$ where for each vertex $i\in [n]$, $\boldsymbol{v}_{k}(i)=\sin(\frac{ki\pi}{n+1})$.
\item
Laplacian: For every $k\in [n]$,  $\theta_{k}=\displaystyle 4\cos^{2}(\frac{k\pi}{2n})$ is a Laplacian eigenvalue of $P_n$ with corresponding eigenvector $\boldsymbol{v}_k$ where for each vertex $i\in [n]$, $\boldsymbol{v}_{k}(i)=\cos(\frac{(n-k)(2i-1)\pi}{2n})$.
\end{enumerate}
\end{theorem}

\begin{example}[\textbf{Eigenvectors of path}]
Let $n\in \mathbb{N}$. We know, from Theorem \ref{pathvec}, that the eigenvalues of $P_{n}$ are
$\lambda_{k}=\displaystyle 2\cos (\frac{k\pi}{n+1})$ for $k\in [n]$ and
$\displaystyle \phi(P_{n},x)=\prod_{k=1}^{n}(x-\displaystyle 2\cos (\frac{k\pi}{n+1}))$.
Since every eigenvector of $P_{n}$ is nonzero on each pendant vertex of the path, (see Remark \ref{pendant}) we choose a pendant vertex as the root and by Theorem  \ref{simpleeig} for the  $\lambda_{k}$-eigenvector $\boldsymbol{\alpha}_{k}$ we have
\[
 \boldsymbol{\alpha}_{k}(i)=\phi(P_{n-i},\lambda_{k})=\prod_{j=1}^{n-i}(\displaystyle 2\cos (\frac{k\pi}{n+1})-\displaystyle 2\cos (\frac{j\pi}{n-i+1})),\,\, i\in[n].
\]
It is obvious that  $ \boldsymbol{\alpha}_{k}(n)=1$.
By comparing with Theorem \ref{pathvec}, we have the following trigonometric identity:
\[\frac{\boldsymbol{\alpha}_{k}(i)}{ \boldsymbol{\alpha}_{k}(n)}=\prod_{j=1}^{n-i}(\displaystyle 2\cos (\frac{k\pi}{n+1})-\displaystyle 2\cos (\frac{j\pi}{n-i+1}))=\displaystyle \frac{\sin (\frac{ki\pi}{n+1})}{\sin (\frac{kn\pi}{n+1})},\,\,i,k\in [n],\]
or equivalently:
\[\prod_{j=1}^{i}(\displaystyle \cos (\frac{k\pi}{n+1})-\displaystyle \cos (\frac{j\pi}{i+1}))=\displaystyle \frac{\sin (\frac{k(i+1)\pi}{n+1})}{2^{i}\sin (\frac{k\pi}{n+1})},\,\,0\leq i\leq n-1,\,k\in [n].\]

\begin{figure}[H]
\centering
\definecolor{qqqqff}{rgb}{0.,0.,1.}
\begin{tikzpicture}[scale=1,line cap=round,line join=round,>=triangle 45,x=1.0cm,y=1.0cm]
\draw (5.0,2)-- (4.1,2);
\draw [style=dotted] (3.2,2.)-- (4.1,2);
\draw [style=dotted](3.2,2.)-- (2.2,2);
\draw (1.4,2)-- (2.2,2);
\begin{scriptsize}
\draw[color=red] (5,2.3) node {$1$};
\draw [color=red] (5.0,2.) circle (3.pt);
\draw [fill=red] (5.0,2.) circle (2.pt);
\draw[color=blue] (4.1,2.3) node {$2$};
\draw [fill=qqqqff] (4.1,2.) circle (2.pt);
\draw[color=blue] (3.2,2.3) node {$i$};
\draw [color=qqqqff] (3.2,2.) circle (3.pt);
\draw [fill=qqqqff] (3.2,2.) circle (2.pt);
\draw[color=blue] (2.2,2.3) node {$n-1$};
\draw [fill=qqqqff] (2.2,2) circle (2.pt);
\draw[color=blue] (1.4,2.3) node {$n$};
\draw [fill=qqqqff] (1.4,2) circle (2.pt);
\end{scriptsize}
\end{tikzpicture}
\caption{Path $P_{n}$ }
\end{figure}
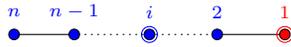 
\end{example}

\begin{example}[\textbf{Laplacian eigenvectors of path}]
By Theorem \ref{pathvec}, the Laplacian eigenvalues of $P_{n}$ are
$\theta_{k}=\displaystyle 4\cos^{2}(\frac{k\pi}{2n})$ for $k\in [n]$.
We choose a pendant vertex as the root and by Theorem  \ref{simpleeig} for the  $\theta_{k}$-eigenvector $\boldsymbol{v}_{k}$ we have
\[
 \boldsymbol{v}_{k}(i)=(-1)^{i-1}\phi(L(P_{n})-\{1,\ldots,i\},\theta_{k}),\,\, i\in[n].
\]
By comparing with Theorem \ref{pathvec}, we have the following identity for $i,k\in [n]$:
\[\frac{\boldsymbol{v}_{k}(i)}{ \boldsymbol{v}_{k}(n)}=
\frac{(-1)^{i-1}\left|\begin{array}{ccccc}
\theta_{k}-2 & 1 & \cdots & 0 & 0\\
 1& \theta_{k}-2& \cdots & 0 & 0 \\
 \vdots & \vdots &\ddots & \vdots & \vdots \\ 
 0 & 0 & \cdots &\theta_{k}-2& 1 \\ 
 0 & 0 & \cdots & 1 &\theta_{k}-1 
\end{array}\right|_{(n-i)\times (n-i)}}{(-1)^{n-1}}=\]
\[(-1)^{i-n}\left|\begin{array}{ccccc}
\displaystyle 2\cos (\frac{k\pi}{n}) & 1 &  \cdots & 0 & 0\\
 1& \displaystyle 2\cos (\frac{k\pi}{n})&  \cdots & 0 & 0 \\
 \vdots & \vdots &\ddots & \vdots & \vdots \\ 
 0 & 0 &  \cdots &\displaystyle 2\cos (\frac{k\pi}{n})& 1 \\ 
 0 & 0 &  \cdots & 1 &\displaystyle 2\cos (\frac{k\pi}{n})+1 
\end{array}\right|=
\displaystyle \frac{\cos(\frac{(n-k)(2i-1)\pi}{2n})}{\cos(\frac{(n-k)(2n-1)\pi}{2n})}.\]
Hence, 
\[\left|\begin{array}{ccccc}
\displaystyle 2\cos (\frac{k\pi}{n}) & 1 &  \cdots & 0 & 0\\
 1& \displaystyle 2\cos (\frac{k\pi}{n})&  \cdots & 0 & 0 \\
 \vdots & \vdots &\ddots & \vdots & \vdots \\ 
 0 & 0 &  \cdots &\displaystyle 2\cos (\frac{k\pi}{n})& 1 \\ 
 0 & 0 &  \cdots & 1 &\displaystyle 2\cos (\frac{k\pi}{n})+1 
\end{array}\right|=
\displaystyle \frac{\sin(\frac{k(2i-1)\pi}{2n})}{\sin(\frac{k(2n-1)\pi}{2n})}.\]
\end{example}

\subsection{Multiple Eigenvalue}
In the following theorem, for multiple eigenvalues of an acyclic symmetric matrix, we give a formula for their corresponding eigenvectors.
\begin{theorem}\label{multipleeig}
Let $n\in\mathbb{N}$ and $A$ be an $n\times n$ acyclic symmetric matrix. Suppose that  $\lambda$ is an eigenvalue of $A$ with multiplicity $k$ and $\lambda$-star set $U=\{u_{1},\ldots,u_{k}\}$. If $U_{i}=U\setminus \{u_i\}$ and $A_i$ is the component of  $A-U_{i}$ that contains $u_i$, for $i\in [k]$,  then 
 $\{\boldsymbol{\alpha}_{i}\}_{i=1}^{k}$ are $k$ independent $\lambda$-eigenvectors of $A$, where 
 \[
 \boldsymbol{\alpha}_{i}(v)=
 \begin{cases}
 W(P_{u_{i},v})\phi(A_{i},\lambda) & v\text{ is in }A_{i},\\
 0 & \text{otherwise}.
 \end{cases}
 \]
\end{theorem}

\begin{proof}
By theorems on star sets, deleting $k-1$ vertices of a $\lambda$-star set from a graph, equivalently removing $k-1$ rows and columns of a $\lambda$-star set from a matrix, we obtain a graph or a matrix with simple eigenvalue $\lambda$. So, by Theorem \ref{simpleeig}, we acquire a $\lambda$-eigenvector corresponding to $A-U_{i}$ and we extend this eigenvector to $A$ by setting zero on $U_{i}$. These $k$ eigenvectors, obviously, are independent and the proof is complete.
\end{proof}

\begin{remark}[\textbf{Matrices of weighted trees}]
The results of this paper do not have conditions on the diagonal entries of acyclic matrices. Hence, we can use them for various matrices associated to weighted trees, such as the adjacency matrix and the Laplacian matrix.

\end{remark}

\section*{Acknowledgements}
The authors are indebted to the School of Mathematics, Institute for Research in Fundamental
Sciences (IPM), Tehran, Iran for the support. The research of the second author was in part supported by grant from IPM (No. 94050116).


\end{document}